\documentclass{agnt}

\usepackage[T1]{fontenc}
\usepackage[latin1]{inputenc}
\usepackage[all]{xy}
\usepackage{graphics}
\usepackage{amscd}
\usepackage{amssymb}
\usepackage{amsthm}
\usepackage{enumitem}
\usepackage{comment}

\newtheorem{thm}{Theorem}[section]

\newtheorem{pro}[thm]{Proposition}
\newtheorem{lem}[thm]{Lemma}
\newtheorem{cor}[thm]{Corollary}

\newtheorem*{rem*}{Remarks}
\newtheorem{rems}[thm]{Remark}
\newtheorem*{conj*}{Conjecture}

\DeclareMathOperator{\C}{\mathbb{C}}

\DeclareMathOperator{\F}{\mathbb{F}}

\DeclareMathOperator{\Q}{\mathbb{Q}}
\DeclareMathOperator{\R}{\mathbb{R}}
\DeclareMathOperator{\Z}{\mathbb{Z}}

\DeclareMathOperator{\Spec}{Spec}

\DeclareMathOperator{\Tr}{Tr}

\DeclareMathOperator{\re}{Re}

\DeclareMathOperator{\id}{id}

\newcommand{\mrm}{\mathrm}
\newcommand{\mbb}{\mathbb}

\title{Abelian varieties and transversal index theorems}
\subject{11M36}{58J22}
\author{Ouidad Filali}
\address{Mathematisches Institut, Westf\"alische Wilhelms-Universit\"at, Einsteinstrasse 62, 48149 M\"unster, Germany.}
\email{ouidadf@hotmail.com}

\author{Francesco Lemma}
\address{Institut math\'ematique de Jussieu-Paris Rive Gauche, UMR 7586, B\^atiment Sophie Germain, Case 7012, 75205 Paris Cedex 13.}
\email{francesco.lemma@imj-prg.fr}

\begin{document}

\maketitle

\begin{abstract}
We interpret the "explicit formula" in the sense of analytic number theory for the zeta function of an ordinary abelian variety of dimension $g$ over a finite field as a transversal index theorem on a Riemannian foliated space of dimension $2g+1$. This generalizes a work of Deninger for elliptic curves.
\end{abstract}

\section{Introduction}

In the search for an understanding of the properties of zeta and $L$-functions, Deninger has developed analogies between the theory of dynamical systems on certain foliated spaces and arithmetic geometry \cite{deninger2}, \cite{deninger1}, \cite{deninger3}. For example, to a $d$-dimensional regular scheme $\mathcal{X}$ of finite type over $\Spec \Z$ should correspond a triple $(X, \mathcal{F}, \phi^t)$ where $X$ is a certain $(2d+1)$-dimensional space endowed with a one codimensional lamination $\mathcal{F}$ and a flow $\phi^t$ whose orbits are transversal to the leaves and such that the closed points $x \in \mathcal{X}$ should correspond to the closed orbits $\gamma$ of $\phi^t$. In the simple case where $\mathcal{X}$ is an elliptic curve over a finite field, Deninger constructed such a three dimensional foliated dynamical system $(X, \mathcal{F}, \phi^t)$ and interpreted the explicit formula for the zeta function of the elliptic curve as a transversal index theorem on $(X, \mathcal{F}, \phi^t)$ \cite{deninger}. In this paper, we will be concerned in generalizing Deninger's construction to higher dimension.\\

Let $A_0/k$ be an abelian variety of dimension $g$ over the finite field $k=\F_q$, let $|A_0|$ the set of closed points of $A_0$ and, for $x \in |A_0|$, let $\deg(x)=[k(x):\F_q]$ be the degree of the residue field of $x$. The zeta function 
$$
\zeta_{A_0}(s)=\prod_{x \in |A_0|} \frac{1}{1-q^{-s\deg(x)}}
$$
converges for $\re s > g$. Let $l$ be a prime different from the characteristic of $k$, let $T_l(A_0)$ be the $l$-adic Tate module of $A_0$, a free $\Z_l$-module of finite rank $2g$, and let $\mu_1, \ldots, \mu_g$ be the eigenvalues of the arithmetic Frobenius endomorphism of $T_l(A_0)$. Then we have
$$
\zeta_{A_0}(s)=\prod_{j=0}^{2g} P_j(q^{-s})^{(-1)^j}
$$
where $
P_j(X)=\prod(1-\mu_{i_1}\ldots \mu_{i_j}X)$, the product being taken over all $j$-uples $(i_1, \ldots, i_j)$ such that $1 \leq i_1 < \ldots < i_j \leq 2g$. Moreover, the functional equation
$$
\zeta_{A_0}(s)=\pm \zeta_{A_0}(g-s)
$$
is satisfied. The explicit formula we are interested in is the following. Given a test function $\alpha \in \mathcal{C}^{\infty}_0(\R)$ we have
\begin{equation} \label{EF}
\sum_{j=1}^{2g} \sum_{\rho_j}(-1)^{j} \Phi(\rho_j) = \log q \sum_{x \in |A_0|}\deg(x) \sum_{k \geq 1} \alpha(k \deg(x) \log q)
\end{equation}
\begin{eqnarray*}
\,\,\,\,\,\,\,\,\,\,\,\,\,\,\,\,\,\,\,\,\,\,\,\,\,\,\,\,\,\,\,\,\,\,\,\,\,\,\,\,\,\,\,\,\,\,\,\,\,\,\,\,\,\,\,\,\,\,\,\,\,\,\,\,\,\,\,\,\,\,\,\,\,\,\,\,\,\,\,\,\,\,\,\,\,\,\,\,\,\,\,\,\,\,+ \log q \sum_{x \in |A_0|}\deg(x) \sum_{k \leq -1} q^{kg\deg(x)} \alpha(k \deg(x) \log q)
\end{eqnarray*}
where $\Phi(s)$ is the function defined by the integral
$$
\Phi(s)=\int_{\R}e^{ts} \alpha(t) dt
$$
and where the sums $\sum_{\rho_j}$ in the left hand term are indexed by the zeroes of $P_j(q^{-s})$. This equality can be proved in a similar, and in fact easier way as in the standard case \cite{barner} 4. Under the assumption that $A_0$ is ordinary, we are going to interpret the explicit formula as a transversal index theorem on a $(2g+1)$-dimensional Riemannian foliated space $S(A_0)$. See Cor. \ref{main} for a precise statement.Transversal index theory is concerned with differential operators or complexes of such operators, which are elliptic in the directions transversal to the orbits of an action by a Lie group, which in our case will simply be $\R^\times_+$. For a short introduction to such a theory, the reader is referred to \cite{deninger} 2.\\

\textbf{Acknowledgements.} The main ideas of the present article have their origin in the 2007 PhD thesis of the first named author under supervision of Christopher Deninger \cite{filali}. Recently, the second named author worked on the subject idependently and brought a new idea. The second named author would like to thank Eric Urban for his invitation to Columbia University, where part of this work has been done and Michael Harris for support. Finally, we would like to thank Christopher Deninger, Frans Oort and Felipe Voloch for enlightening correspondence.

\section{Preliminaries on abelian varieties}\label{preliminaires} In this section, we would like to recall some classical facts about abelian varieties, the Frobenius endomorphism and ordinarity.\\

Let $A/k$ be an abelian variety over a field $k$, of dimension $g>0$ and let $\phi \in \mrm{End}(A)$ be an endomorphism. There exists a polynomial $P_\phi \in \Z[X]$ of degree $2g$ such that for any $t \in \Z$ we have $P_\phi(t)=\deg(\phi-[t])$ where $[t]$ denotes the endomorphism of $A$ given by multiplication by $t$ (see \cite{cornell-silvermann} p. 125). For any prime number $l$, denote by $T_l(A)$ the $l$-adic Tate module of $A$. It follows from \cite{mumford} IV \S 19 Thm. 4 that if $l \neq \mrm{char}(k)$ the polynomial $P_\phi$ is the characteristic polynomial of the endomorphism $T_l(\phi)$ induced by $\phi$ on $T_l(A) \otimes_{\Z_l} \Q_l$. Assume that $k=\mbb{F}_q$ is a finite field with $q$ elements and that $\phi$ is the $q$-th power Frobenius endomorphism, then $\deg(\phi)=q^g$ and hence 
\begin{equation} \label{det}
\det T_l(\phi)=q^g
\end{equation}
for any prime $l$ different from the characteristic of $k$.\\

Let $A_0/k$ be an abelian variety of dimension $g$ over a finite field $k=\F_q$ of characteristic $p$. For any integer $n \geq 0$, let us denote by $G_n$ the kernel of the multiplication by $p^n$ on $A_0$. One says that $A_0$ is ordinary if the following equivalent conditions are satisfied:\\
\indent 1. $G_1(\overline{k}) \simeq (\Z/p\Z)^g$,\\
\indent 2. $G_n(\overline{k}) \simeq (\Z/p^n\Z)^g$ for all $n \geq 0$.\\
A proof of the equivalence of these two conditions can be found in \cite{mumford} \S 15 "the $p$-rank". Let $G_n^0$ be the connected component of the identity in $G_n$ and let $G_n^{et}$ be the largest \'etale quotient of $G_n$. Then, there is an exact sequence
$$
0 \longrightarrow G_n^0 \longrightarrow G_n \longrightarrow G_n^{et} \longrightarrow 0
$$ 
which is split by \cite{tate} (3.7) IV. If $A_0/k$ is ordinary, then by the Serre-Tate theorem \cite{serre-tate}, it admits a canonical lift $\mathcal{A}$ over the ring of Witt vectors $W$ of $k$. This canonical lift is characterized by the following equivalent conditions (see \cite{deligne} 3):\\
\indent 1. The $p$-divisible group of $\mathcal{A}$ is the product of the $p$-divisible groups lifting the connected component of the identity and the largest \'etale quotient of the $p$-divisible group of $A$,\\
\indent 2. Every endomorphism of $A_0$ lifts uniquely to $\mathcal{A}$.

\section{The Riemannian foliated dynamical system $(S(A_0), \mathcal{F}, \phi^t)$} Let $A_0$ be a $g$-dimensionnal ordinary abelian variety over $k=\F_q$, let $p$ be the characteristic of $k$ and let $\phi_0: A_0 \longrightarrow A_0$ be its $q$-th power Frobenius endomorphism. Let $\mathcal{A}$ be the Serre-Tate canonical lift of $A_0$ to $W$ and let $\phi: \mathcal{A} \longrightarrow \mathcal{A}$ be the endomorphism lifting $\phi_0$. Let $K=W[1/p]$ and let $A=\mathcal{A} \otimes_W K$ be the generic fibre of $\mathcal{A}$. In what follows, we fix an embedding $\iota: K \longrightarrow \C$ so that we can consider the complex analytic abelian variety $A(\C)$. Let us denote by $\Gamma$ the first integral homology group $\Gamma=H_1(A(\C), \Z)$, which is a free $\Z$-module of rank $2g$. Via the natural map
$$
\Theta_{\iota}: \mrm{End}_K(A) \longrightarrow \mrm{End}_{\C}(A(\C)),
$$
the Frobenius lift $\phi \otimes K$ induces the endomorphism $$\xi=\Theta_{\iota}(\phi \otimes K)_* \in \mrm{End}_{\Z}(\Gamma).$$
For any prime $l \neq p$, we have a functorial isomorphism $\Gamma \otimes \Z_l = T_l(A)$ of $\Z_l$-modules.  Hence, it follows from (\ref{det}) that
\begin{equation} \label{determinant}
\det(\xi)=q^g
\end{equation}
As the $\R$-vector space $\Gamma \otimes \R$ is identified to the Lie algebra of $A(\C)$,  it has a complex structure for which the endomorphism $\xi \otimes \mrm{id}_{\R}$ is $\C$-linear. In what follows, we will identify the Lie algebra $\Gamma \otimes \R$ to $\C^g$ and we denote by $\xi_{\R} \in \mrm{End}_{\C}(\C^g)$ the endomorphism $\xi \otimes \mrm{id}_{\R}$. Let us introduce the following $\xi$-adic "Tate modules":
\begin{eqnarray*}
T_{\xi} \Gamma &=& \underleftarrow{\lim}_\nu \Gamma/\xi^\nu \Gamma,\\
V_{\xi} \Gamma &=& T_{\xi} \Gamma \otimes \Q.
\end{eqnarray*}
Note that multiplication by $\xi$ defines an automorphism of $V_\xi \Gamma$. We also define an additive subgroup of $\C^g$ by
\begin{eqnarray*}
V &=& \bigcup_{\nu \geq 0} \xi^{-\nu}_{\R} \Gamma
\end{eqnarray*}
The group $V$ acts on $\C^g \times V_\xi \Gamma$ by $v.(z, \hat{v})=(z+v, \hat{v}-v)$ and we denote by $\C^g \times_V V_\xi \Gamma$ the quotient space. The group $q^{\Z}$ acts on $(\C^g \times_V V_\xi \Gamma) \times \R^\times_+$ by $q^\nu ([z, \hat{v}], x)=([\xi^{-\nu}_{\R}z, \xi^{-\nu} \hat{v}], q^\nu x)$ and we denote by
$$
S(A_0)=(\C^g \times_V V_\xi \Gamma) \times_{q^{\Z}} \R^\times_+
$$
the quotient space. The images of the sets $\C^g \times \{\hat{v}\} \times \{x\}$ by the natural projection map $\pi: \C^g \times V_\xi \Gamma \times \R^\times_+ \longrightarrow S(A_0)$ form a partition $\mathcal{F}$ of $S(A_0)$ such that $(S(A_0), \mathcal{F})$ is a foliated space (see \cite{cc} Def. 11.2.12). The tangent spaces to the leaves form a $\R$-vector budle $T\mathcal{F}$ over $S(A_0)$. We define a flow $\phi^t$ on $(S(A_0), \mathcal{F})$ by $$\phi^t([z, \hat{v}, x])=[z, \hat{v}, e^tx].$$ Clearly, this flow sends each leaf to another leaf.\\

As the complex torus $\C^g/\Gamma$  is the set of complex points of an abelian variety, it admits a non-degenerate Riemann form $H$ (see \cite{cornell-silvermann} Thm. A p. 85). This means that there exists a positive definite hermitian form $H$ on $\C^g$ such that the restriction of the imaginary part $\mrm{Im} H$ to $\Gamma$ is integral valued. Let us fix such an hermitian form $H$ once and for all and let us denote by $\Psi: \Gamma \otimes \Gamma \longrightarrow \Z$ the induced pairing. Note that, for any $\eta_1, \eta_2 \in \C^g$, we have $\mrm{Re} H(\eta_1, \eta_2)=\mrm{Im} H(i\eta_1, \eta_2)$. 

\begin{pro} \label{metrique} The Riemannian metric on the bundle $T\C^g \times V_\xi \Gamma \times \R^\times_+$ over the space $\C^g \times V_\xi \Gamma \times \R^\times_+$ given by the formula $$
\tilde{g}_{[z, \hat{v}, x]}(\eta_1, \eta_2)=x \mrm{Re} H(\eta_1, \eta_2)$$
induces a metric $g$ along the leaves of $(S(A_0), \mathcal{F})$ such that $$(\phi^t)^*g=e^t g.$$
\end{pro}

\begin{proof} To prove the first statement, we need to show that for every $\nu \in \Z$ one has
$$
\mrm{Re} H(\xi_{\R}^{\nu}(\eta_1), \xi_{\R}^{\nu}(\eta_2))=q^{\nu}\mrm{Re} H (\eta_1, \eta_2).
$$
Let $l \neq p$ be a prime. Identify $\Z_l$ to $\Z_l(1)$ via the compatible system of primitive $l^n$-th roots of unity $(e^{2\pi i/l^n})$. Then the Weil pairing $$\Psi_l: T_l(A) \otimes T_l(A) \longrightarrow \Z_l(1)$$ is minus the $\Z_l$-linear extension of $\Psi$ (\cite{mumford} Thm. 1 p. 237). As a consequence, for any $\gamma_1, \gamma_2 \in \Gamma$, we have
$$
\Psi(\xi^{\nu}(\gamma_1), \xi^{\nu}(\gamma_2))=-\Psi_l(\xi^{\nu}(\gamma_1), \xi^{\nu}(\gamma_2))=-q^{\nu}\Psi_l(\gamma_1, \gamma_2)=q^{\nu}\Psi(\gamma_1. \gamma_2)
$$
This implies that $\mrm{Im} H(\xi_{\R}^{\nu}(\eta_1), \xi_{\R}^{\nu}(\eta_2))=q^{\nu} \mrm{Im} H(\eta_1, \eta_2)$ for any $\eta_1, \eta_2 \in \C^g=\Gamma \otimes \R$. As a consequence we have
\begin{eqnarray*}
\mrm{Re} H(\xi_{\R}^{\nu}(\eta_1), \xi_{\R}^{\nu}(\eta_2)) &=& \mrm{Im} H(i \xi_{\R}^{\nu}(\eta_1), \xi_{\R}^{\nu}(\eta_2))\\
&=& \mrm{Im} H(\xi_{\R}^{\nu}(i \eta_1), \xi_{\R}^{\nu}(\eta_2))\\
&=& q^{\nu} \mrm{Im} H(i \eta_1, \eta_2)\\
&=& q^{\nu} \mrm{Re} H(\eta_1, \eta_2).
\end{eqnarray*} 
Let us prove the second statement. Let $t \in \R$, $x \in \R^\times_+$ and assume that $e^tx=q^\nu x'$ for some $\nu \in \Z$ and $x' \in \R^\times_+$. Then for any $z \in \C^g$ and $\hat{v} \in V_\xi \Gamma$, we have 
$$
\phi^t([z, \hat{v}, x])=[\xi_{\R}^{\nu}(z), \xi^{\nu}(\hat{v}), x']
$$
and the tangent map $T_{[z, \hat{v}, x]}\phi^t: T_{[z, \hat{v}, x]} \mathcal{F} \longrightarrow T_{[\xi_{\R}^{\nu}(z), \xi^{\nu}(\hat{v}), x']} \mathcal{F}$ sends $\eta \in \C^g$ to $\xi^{\nu}(\eta)$. As a consequence
\begin{eqnarray*}
(\phi^{t*}g)(\eta_1, \eta_2) &=& x'\re H(T_{[z, \hat{v}, x]}\phi^t(\eta_1), T_{[z, \hat{v}, x]}\phi^t(\eta_2))\\
&=& x' \re H(\xi_{\R}^{\nu}(\eta_1), \xi_{\R}^{\nu}(\eta_2))\\
&=& q^{\nu}x' \re H(\eta_1, \eta_2)\\
&=& e^tg(\eta_1, \eta_2).
\end{eqnarray*}
\end{proof}

\begin{lem} \label{basis} There exists an orthonormal basis $
\left( \frac{\partial}{\partial x_1}, \frac{\partial}{\partial y_1}, \ldots, \frac{\partial}{\partial x_{g}}, \frac{\partial}{\partial y_{g}} \right)$ of $(\Gamma \otimes \R, \re H)$ such that the vectors $\frac{\partial}{\partial z_j}=\frac{\partial}{\partial x_j}+\frac{\partial}{\partial y_j} \otimes i \in \Gamma \otimes \C$ are eigenvectors of $\xi_{\C}=\xi \otimes \id_{\C}$.
\end{lem}

\begin{proof} We have explained in the proof of the lemma above that $q^{-\frac{1}{2}} \xi_{\R}$ is an orthogonal automorphism of $(\Gamma \otimes \R, \re H)$. By elementary linear algebra, there exists an orthonormal basis $
\left( \frac{\partial}{\partial x_1}, \frac{\partial}{\partial y_1}, \ldots, \frac{\partial}{\partial x_{g}}, \frac{\partial}{\partial y_{g}} \right)$ of $(\Gamma \otimes \R, \re H)$ in which the matrix of $q^{-\frac{1}{2}} \xi_{\R}$ is diagonal by blocks of size one of the form $(\pm 1)$ and of size two of the form $\begin{pmatrix}
\cos \theta & -\sin \theta\\
\sin \theta & \cos \theta \\
\end{pmatrix}$. In fact, the blocks of size one do not occur because $\pm q^{\frac{1}{2}}$ is not an eigenvalue of $\xi$ by \cite{deligne} proof of Thm. 7. (A). Then the basis $
\left( \frac{\partial}{\partial x_1}, \frac{\partial}{\partial y_1}, \ldots, \frac{\partial}{\partial x_{g}}, \frac{\partial}{\partial y_{g}} \right)$ satisfies the statement of the lemma.
\end{proof}

The following results will be useful later: let us denote by $\Gamma_p$ the $p$-adic Tate module $\Gamma \otimes \Z_p$. According to \cite{deligne} 6.1 and 6.2, we have a decomposition into $\Z_p$-modules $\Gamma_p=\Gamma'_p \oplus \Gamma''_p$ such that $\xi(\Gamma'_p)=\Gamma'_p$ and $\xi(\Gamma''_p)=q \Gamma''_p$. It follows from (\ref{determinant}) that $\Gamma/\xi^\nu \Gamma$ is a finite abelian $p$-group for every $\nu \geq 0$ and hence that $T_{\xi} \Gamma=\underleftarrow{\lim}_\nu \Gamma_p/\xi^\nu \Gamma_p$.  As a consequence, we have 
\begin{equation} \label{technique2}
T_\xi \Gamma=\underleftarrow{\lim}_\nu \Gamma''_p/\xi^\nu \Gamma''_p= \underleftarrow{\lim}_\nu \Gamma''_p/q^\nu \Gamma''_p=\Gamma''_p.
\end{equation}
We define the quotient space $\C^g \times_{\Gamma} T_\xi \Gamma$ similarly as $\C^g \times_V V_\xi \Gamma$.

\begin{lem} \label{technique} The natural inclusion $\C^g \times T_\xi \Gamma \longrightarrow \C^g \times V_\xi \Gamma$ induces a canonical isomorphism $\C^g \times_\Gamma T_\xi \Gamma \simeq \C^g \times_V V_\xi \Gamma$ which is equivariant with respect to the diagonal action of $\xi$ on both sides. 
\end{lem}

\begin{proof} The natural inclusion $\Gamma[1/p] \longrightarrow V_\xi \Gamma$, which is dense, factors through the inclusion $\Gamma''_p[1/p] \longrightarrow V_\xi \Gamma$. Hence, for any $\hat{v} \in V_\xi \Gamma$ there exists $p^{-\alpha} \gamma$, where $\alpha \in \Z$ and $\gamma \in \Gamma$ such that $\hat{v}-p^{-\alpha} \gamma \in T_\xi \Gamma$ and we can assume that $\gamma \in \Gamma \cap \Gamma''_p$. Hence we have $\xi_{\R}^\alpha(p^{-\alpha} \gamma) \in \Gamma$ which means that $p^{-\alpha} \gamma \in V$. This concludes the proof.
\end{proof}

\begin{cor} \label{compact} The topological space $S(A_0)$ is compact.
\end{cor}

\begin{proof} The natural projection $\C^g \times_{\Gamma} T_{\xi} \Gamma \longrightarrow \C^g/\Gamma$ is locally trivial with compact fiber and compact target. This implies that $\C^g \times_\Gamma T_\xi \Gamma$ is compact. Similarly, we have that $(\C^g \times_{\Gamma} T_\xi \Gamma) \times_{q^{\Z}} \R^\times_+$ is compact. But this space is nothing but $S(A_0)$ by the previous lemma.
\end{proof}

\section{$L^2$-harmonic forms along the leaves}

In this section, we work with a fixed orthonormal basis of $(\C^g, \re H)$ given by Lem. \ref{basis} and we denote by $(dx_1, dy_1, \ldots, dx_g, dy_g)$ the dual basis.

\begin{lem} \label{mesure} Let $\mu_\xi$ be a Haar measure on the locally compact abelian group $V_\xi \Gamma$.\\
\indent $(i).$ The measure $\prod_{j=1}^g dx_j dy_j \otimes \mu_\xi \otimes \frac{dx}{x}$ induces a measure $\mu$ on $S(A_0)$.\\
\indent $(ii).$ The measure $\mu$ is invariant under the action of $\phi^t$.
\end{lem}

\begin{proof} (i). We need to check that for any $\nu \in \Z$ and any Borel subset $A$ of $V_\xi \Gamma$ we have
$$
\mu_{\xi}(\xi^{\nu}(A))=|\det{\xi}^{\nu}|^{-1}\mu_{\xi}(A)=q^{-g\nu}\mu_{\xi}(A).
$$
As $A \longmapsto \mu_{\xi}(\xi^{\nu}(A))$ is also a Haar measure on $V_{\xi} \Gamma$, it is enough to verify the above equality for $A=T_{\xi} \Gamma$. But it follows from the fact that, for any $\nu \geq 0$, we have $T_\xi \Gamma/\xi^\nu T_\xi \Gamma=\Gamma/\xi^\nu \Gamma$ has $\det(\xi^\nu)=q^{g\nu}$ elements. (ii). Trivial.
\end{proof}

For $0 \leq j \leq 2g$ let $\mathcal{A}^j_{\mathcal{F}}(S(A_0))$ be the $\R$-vector space of sections of the vector bundle $\bigwedge^j T^*\mathcal{F}$ which are smooth along the $\mathcal{F}$-leaves and continuous transversally (see \cite{moore-schochet} III Def. 3.2). Denote by $$d^j_{\mathcal{F}}: \mathcal{A}^j_{\mathcal{F}}(S(A_0)) \longrightarrow \mathcal{A}^{j+1}_{\mathcal{F}}(S(A_0))$$ the leafwise exterior derivative. These form the de Rham complex along $\mathcal{F}$
$$
\begin{CD}
0 @>>> \mathcal{A}^0_{\mathcal{F}}(S(A_0)) @>>> \ldots @>>> \mathcal{A}^{2g}_{\mathcal{F}}(S(A_0)) @>>> 0.
\end{CD}
$$
The metric $g$ induces a metric $g_j$ on the vector bundle $\bigwedge^j T^*\mathcal{F}$ in the following standard way: if $e_1, \ldots, e_{2g}$ is an orthonormal basis of $T_{[z, \hat{v}, x]} \mathcal{F}$ and $e_1^*, \ldots, e_{2g}^*$ is the dual basis, then the vectors $e_{i_1}^* \wedge \ldots \wedge e_{i_j}^*$ for $1 \leq i_1 < \ldots < i_j \leq 2g$ form an orthonormal basis of $\bigwedge^j T^*_{[z, \hat{v}, x]}\mathcal{F}$. Because $S(A_0)$ is compact by Cor. \ref{compact} we can define the scalar product on $\mathcal{A}^j_{\mathcal{F}}(S(A_0))$ by
$$
\langle \omega, \omega' \rangle=\int_{S(A_0)} g_j(\omega, \omega') d\mu.
$$
Let $d_{\mathcal{F}}^{j\dagger}: \mathcal{A}^j_{\mathcal{F}}(S(A_0)) \longrightarrow \mathcal{A}^{j-1}_{\mathcal{F}}(S(A_0))$ be the formal adjoint of $d^j_{\mathcal{F}}$. We also denote as usual by $\Delta^j_\mathcal{F}=d^j_\mathcal{F}d^{j \dagger}_\mathcal{F}+d^{j \dagger}_\mathcal{F} d^j_\mathcal{F}$ the Laplacian.

\begin{lem} \label{formulae} Let 
$$
\omega=\sum_{1 \leq i_1< \ldots < i_j \leq 2g} \alpha_{i_1 \ldots i_j} du_{i_1} \wedge \ldots \wedge du_{i_j} \in  \mathcal{A}^j_{\mathcal{F}}(S(A_0)).
$$
Then
\begin{eqnarray*}
d_{\mathcal{F}}^j \omega &=& \sum_{1 \leq i_1 < \ldots < i_{j+1} \leq 2g} \sum_{1 \leq k \leq j+1}(-1)^{k-1} \frac{\partial \alpha_{i_1 \ldots \hat{i}_k \ldots i_{j+1}}}{\partial u_{i_k}} du_{i_1} \wedge \ldots \wedge du_{i_{j+1}},\\
d_{\mathcal{F}}^{j\dagger} \omega &=& -x^{-1} \sum_{1 \leq i_1< \ldots < i_j \leq 2g} \sum_{1 \leq k \leq j}(-1)^{k-1}\frac{\partial \alpha_{i_1 \ldots i_j}}{\partial u_{i_k}} du_{i_1} \wedge \ldots \wedge \hat{du}_{i_k} \ldots \wedge du_{i_j},\\
\Delta^j_\mathcal{F} \omega &=& -x^{-1} \sum_{1 \leq i_1< \ldots < i_j \leq 2g} \sum_{1 \leq k \leq j} \frac{\partial^2 \alpha_{i_1 \ldots i_j}}{\partial u_k^2} du_{i_1} \wedge \ldots \wedge du_{i_j}.
\end{eqnarray*}
\end{lem}

\begin{proof} The first statement follows from an easy standard computation. To prove the second statement let 
$$
\eta=\sum_{1 \leq i_1<\ldots<i_{j-1}  \leq 2g}\beta_{i_1 \ldots i_{j-1} } du_{i_1} \wedge \ldots \wedge du_{i_{j-1}} \in \mathcal{A}^{j-1}_{\mathcal{F}}(S(A_0)).
$$
Then, one has
$$
\langle d_{\mathcal{F}}^{j\dagger} \omega, \eta \rangle = \langle \omega, d^{j-1}_{\mathcal{F}} \eta \rangle
$$
\begin{eqnarray*}
&=& \sum_{1 \leq i_1< \ldots < i_j \leq 2g} \sum_{1 \leq k \leq j}(-1)^{k-1} \langle \alpha_{i_1 \ldots i_j} du_{i_1} \wedge \ldots \wedge du_{i_j},  \frac{\partial \beta_{i_1 \ldots \hat{i}_k \ldots i_{j}}}{\partial u_{i_k}} du_{i_1} \wedge \ldots \wedge du_{i_{j}} \rangle\\
&=& \sum_{1 \leq i_1< \ldots < i_j \leq 2g} \sum_{1 \leq k \leq j}(-1)^{k-1} \int_{S(A_0)} x^{-j} \alpha_{i_1 \ldots i_j}\frac{\partial \beta_{i_1 \ldots \hat{i}_k \ldots i_{j}}}{\partial u_{i_k}} d\mu\\
&=& -\sum_{1 \leq i_1< \ldots < i_j \leq 2g} \sum_{1 \leq k \leq j}(-1)^{k-1} \int_{S(A_0)} x^{-j} \frac{\partial \alpha_{i_1 \ldots i_j}}{\partial u_{i_k}} \beta_{i_1 \ldots \hat{i}_k \ldots i_{j}} d\mu \\
&=&  \left \langle -x^{-1} \sum_{1 \leq i_1< \ldots < i_j \leq 2g} \sum_{1 \leq k \leq j}(-1)^{k-1}\frac{\partial \alpha_{i_1 \ldots i_j}}{\partial u_{i_k}} du_{i_1} \wedge \ldots \wedge \hat{du}_{i_k} \ldots \wedge du_{i_j}, \eta \right \rangle
\end{eqnarray*}
which proves the statement. The proof of the last statement is left to the reader.
\end{proof}

Let $\mathcal{A}^j_{\mathcal{F}, L^2}(S(A_0))$ be the $L^2$-completion of $\mathcal{A}^j_{\mathcal{F}}(S(A_0))$, view $d_{\mathcal{F}}^j$ as an unbounded operator on $\mathcal{A}^j_{\mathcal{F}, L^2}(S(A_0))$ and define $\tilde{d}_{\mathcal{F}}^j$ to be the closed unbounded operator $\tilde{d}_{\mathcal{F}}^j={d}_{\mathcal{F}}^{j\dagger *}$. Define
$$
\mrm{Harm}^j_{L^2}(S(A_0))=\ker \tilde{d}^j_{\mathcal{F}} \cap \ker \tilde{d}_{\mathcal{F}}^{j*}.
$$
This is a Hilbert space, as a closed subspace of $\mathcal{A}^j_{\mathcal{F}, L^2}(S(A_0))$.\\

We introduce the following notation:
\begin{eqnarray*}
\mathcal{A}^j_{\mathcal{F}, L^2}(S(A_0))_{\C} &=& \mathcal{A}^j_{\mathcal{F}, L^2}(S(A_0)) \otimes_{\R} \C,\\
\mrm{Harm}^j_{L^2}(S(A_0))_{\C} &=& \mrm{Harm}^j_{L^2}(S(A_0)) \otimes_{\R} \C.
\end{eqnarray*}
Let us denote by $(d\tau_1, \ldots, d\tau_{2g})$ the basis of $(\Gamma \otimes \C)^*$ which is dual to the basis $(\frac{\partial}{\partial z_1}, \frac{\partial}{\partial \overline{z}_1}, \ldots, \frac{\partial}{\partial z_g}, \frac{\partial}{\partial \overline{z}_g})$ (see Lem. \ref{basis}). We endow $\mrm{Harm}^j_{L^2}(S(A_0))_{\C}$ with the hermitian scalar product induced by the scalar product on $\mrm{Harm}^j_{L^2}(S(A_0))$. Note that any element $\omega$ of $\mathcal{A}^j_{\mathcal{F}, L^2}(S(A_0))_{\C}$ is written as 
$$
\omega=\sum_{1 \leq i_1< \ldots < i_j \leq 2g} \alpha_{i_1 \ldots i_j} d\tau_{i_1} \wedge \ldots \wedge d\tau_{i_j}
$$
where $\alpha_{i_1 \ldots i_j}$ are $\C$-valued $L^2_{loc}$ functions on $\C^g \times V_{\xi} \Gamma \times \R^\times_+$ satisfying the following invariance properties:
\begin{eqnarray}
\alpha_{i_1 \ldots i_j}(z+v, \hat{v}-v, x) &=& \alpha_{i_1 \ldots i_j}(z, \hat{v},x),\\
\alpha_{i_1 \ldots i_j}(z, \hat{v},q^\nu x) d\tau_{i_1} \wedge \ldots \wedge d\tau_{i_j} &=& \alpha_{i_1 \ldots i_j}(z, \hat{v},x) \xi^{\nu *}_{\C}(d\tau_{i_1} \wedge \ldots \wedge d\tau_{i_j})
\end{eqnarray}
The proof of the following result is inspired from the beginning of the one of \cite{deninger} Thm. 4.1. 

\begin{lem} \label{harmo} Let
$$
\omega=\sum_{1 \leq i_1< \ldots < i_j \leq 2g} \alpha_{i_1 \ldots i_j} d\tau_{i_1} \wedge \ldots \wedge d\tau_{i_j} \in \mathcal{A}^j_{\mathcal{F}, L^2}(S(A_0))_{\C}.
$$
Then $\omega \in \mrm{Harm}^j_{L^2}(S(A_0))_{\C}$ if and only if for any $x \in \R^\times_+$ the functions defined by $\C^g \times V_\xi \Gamma \longrightarrow \C$, $(z, \hat{v}) \longmapsto \alpha_{i_1 \ldots i_j}(z, \hat{v}, x)$ are constant.
\end{lem}

\begin{proof} The fact that the condition is sufficient follows from the first two statements of Lem. \ref{formulae}. Let us prove that the condition is necessary. Fix $x \in \R^\times_+$. Denote by $\overline{M}$ the abelian group $\C^g \times_V V_{\xi} \Gamma$. As $\alpha_{i_1 \ldots i_j}$ is in $L^2_{loc}$ and $\overline{M}$ is compact (Cor. \ref{compact}), the function $(z, \hat{v}) \longmapsto \alpha_{i_1 \ldots i_j}(z, \hat{v}, x)$ is in $L^2(\overline{M})$. The character group of $\overline{M}$ is
$$
\left\{ \chi \otimes \chi' \,|\, \chi \in (\C^g)^\vee, \chi' \in (V_{\xi} \Gamma)^\vee, \chi|_V=\chi'|_V \right\}
$$
hence, by Fourier theory, one has the equality
$$
\alpha_{i_1 \ldots i_j}(z, \hat{v}, x) = \sum_{\chi|_V=\chi'|_V} a_{\chi, \chi'} \chi(z) \chi'(\hat{v})
$$
in $L^2(\overline{M})$. Any character $\chi$ of $\C^g$ is of the form $$\chi(z)=\chi_w(z)=\prod_{1 \leq k \leq g} \exp(w_k z_k- \overline{w}_k \overline{z}_k)$$ for a uniquely determined $w=(w_1, \ldots, w_g) \in \C^g$. Because $$\omega \in \mrm{Harm}^j_{L^2}(S(A_0)),$$ we have $\langle \omega, \Delta^j_{\mathcal{F}} \beta \rangle =0$ for every $\beta \in \mathcal{A}^j_{\mathcal{F}}(S(A_0))$. Since distributional derivatives commute with convergent series of distributions it follows that
$$
 \sum_{\chi|_V=\chi'|_V} a_{\chi, \chi'} \Delta(\chi(z)) \chi'(\hat{v})=0
$$ where $\Delta$ is the usual Laplacian on $\C^g=\R^{2g}$ (see the third statement of Lem. \ref{formulae}). As $\Delta \chi_w(z)=-|w|^2 \chi_w(z)$, the coefficients $a_{\chi, \chi'}$ corresponding to the non-trivial $\chi$ are zero. This implies that the $\chi'$ such that $a_{\chi, \chi'} \neq 0$ are trivial on $V$ hence are trivial on $V_{\xi} \Gamma$ by density of $V$ in $V_{\xi} \Gamma$. This shows that $\alpha_{i_1 \ldots i_j}(z, \hat{v}, x)$ does not depend on $(z, \hat{v}) \in \C^g \times V_{\xi} \Gamma$.
\end{proof}

Recall that we work with a basis of the $\R$-vector space $\Gamma \otimes \R=\C^g$ given by Lem. \ref{basis}. In particular,  each $d\tau_i$ is an eigenvector of $\xi_{\C}$. Let us denote by $\mu_i$ the corresponding eigenvalue and let us fix a branch $\log_q$ of the complex logarithm for the basis $q$.

\begin{cor} The Hilbert space $\mrm{Harm}^j_{L^2}(S(A_0))_{\C}$ has an orthonormal basis consisting of the
$$
x^{\frac{2 \pi i \nu}{\log q}+\log_q(\prod_{k=1}^j \mu_{i_k})} d\tau_{i_1} \wedge \ldots \wedge d\tau_{i_j}
$$
where $\nu \in \Z$ and $1 \leq i_1 < \ldots<  i_j \leq 2g$.
\end{cor}

\begin{proof}
According to Lem. \ref{harmo} an element $$\alpha_{i_1 \ldots i_j} d\tau_{i_1} \wedge \ldots \wedge d\tau_{i_j}$$ of $\mathcal{A}^j_{\mathcal{F}, L^2}(S(A_0))_{\C}$ belongs to $\mrm{Harm}^j_{L^2}(S(A_0))_{\C}$ if and only if $\alpha_{i_1 \ldots i_j}$ does not depend on the variables in $\C^g$ and $V_{\xi} \Gamma$. By (6), we have
$$
\alpha_{i_1 \ldots i_j}(q^\nu x)=\left(\prod_{k=1}^j \mu_{i_k} \right)^\nu \alpha_{i_1 \ldots i_j}(x)
$$
for any $\nu \in \Z$ and $x \in \R^\times_+$. As a consequence, the function 
$$
x \longmapsto \left(\prod_{k=1}^j \mu_{i_k} \right)^{-\log_q x} \alpha_{i_1 \ldots i_j}(x)=x^{-\log_q(\prod_{k=1}^j \mu_{i_k})} \alpha_{i_1 \ldots i_j}(x)
$$ is in $L^2(\R_+^\times/q^{\Z}, \C)$. Then the fact that the family $$\left(x^{\frac{2 \pi i \nu}{\log q}+\log_q(\prod_{k=1}^j \mu_{i_k})} d\tau_{i_1} \wedge \ldots \wedge d\tau_{i_j} \right)_{\nu \in \Z, 1 \leq i_1<\ldots<i_j \leq 2g}$$ is a basis follows from the fact that $\left(x^{\frac{2 \pi i \nu}{\log q}} \right)_{\nu \in \Z}$ is a basis of $L^2(\R_+^\times/q^{\Z}, \C)$. The orthogonality statement is obvious. The fact that the vectors have norm equal to one follows from the fact that the eigenvalues $\mu$ of $\xi$ verify $\mu \overline{\mu}=q$ by Weil.
\end{proof}

Since $\mu$ is $\phi^t$-invariant, for $\omega, \omega' \in \mathcal{A}^j_{\mathcal{F}}(S(A_0))$ we have $$\langle \phi^{t *}\omega, \phi^{t *}\omega' \rangle=e^{jt}\langle \omega, \omega' \rangle.$$
Hence $e^{-jt/2}\phi^{t *}$ is an orthogonal operator on the real Hilbert space $\mathcal{A}^j_{\mathcal{F}}(S(A_0))$. It follows that $\phi^{t*}$ has a unique extension to $\mathcal{A}^j_{\mathcal{F}, L^2}(S(A_0))$, still denoted by $\phi^{t*}$, such that $e^{-jt/2} \phi^{t*}$ is orthogonal. As a consequence $\phi^{t*}$ commutes with $\tilde{d}^j_{\mathcal{F}}$ and with $\tilde{d}_{\mathcal{F}}^{j*}$. In particular $\phi^{t *}$ leaves $\mrm{Harm}^j_{L^2}(S(A_0))$ invariant. Let $\alpha \in \mathcal{C}^\infty_0(\R)$ be a test function on $\R$. Consider the bounded operator
$$
S_j(\alpha)=\int_{\R} \alpha(t) \phi^{t*}_{\C}dt
$$
on the complex Hilbert space $\mrm{Harm}^j_{L^2}(S(A_0))_{\C}$ where $\phi^{t*}_{\C}=\phi^{t*} \otimes \id_{\C}$.

\begin{pro} \label{eigenvalues} Let $0 \leq j \leq 2g$. For every $\alpha \in \mathcal{C}^\infty_0(\R)$, the operator $S_j(\alpha)$ is of trace class and its trace is given by
$$
\Tr\left( S_j(\alpha)\,|\, \mrm{Harm}^j_{L^2}(S(A_0))_{\C} \right)=\sum_{\rho_j} \Phi(\rho_j)
$$
where the sum is indexed by the zeroes of $P_j(q^{-s})$ and where $$\Phi(s)=\int_{\R}e^{ts}\alpha(t)dt.$$
\end{pro}

\begin{proof} For every $\nu \in \Z$ and every $1 \leq i_1<\ldots<i_j \leq 2g$, the basis vector $$x^{\frac{2 \pi i \nu}{\log q}+\log_q(\prod_{k=1}^j \mu_{i_k})} du_{i_1} \wedge \ldots \wedge du_{i_j}$$ is an eigenvector of $\phi^{t*}$ with eigenvalue $e^{t\rho_j}$ where $$\rho_j=\frac{2 \pi i \nu}{\log q}+\log_q(\prod_{k=1}^j \mu_{i_k}).$$ These $\rho_j$ are precisely the zeroes of $P_j(q^{-s})$. Hence we need to show that
$$
\sum_{\nu \in \Z, 1 \leq i_1<\ldots<i_j \leq 2g} \left| \int_{\R} e^{t \rho_j} \alpha(t) dt \right|  < +\infty.
$$
But this follows from straightforward estimates of Fourier coefficients.
\end{proof}

\section{A transversal index computation on $(S(A_0), \mathcal{F}, \phi^t)$}

The transversal index of the de Rham complex of $S(A_0)$ along $\mathcal{F}$ is defined as the distribution
$$
\mrm{Ind}_t(d_{\mathcal{F}}): \alpha \longmapsto \mrm{Ind}_t(d_{\mathcal{F}})(\alpha)=\sum_{j=0}^{2g}(-1)^j \Tr\left( S_j(\alpha)\,|\, \mrm{Harm}^j_{L^2}(S(A_0))_{\C} \right).
$$
Then by Prop. \ref{eigenvalues}, the left hand side of the explicit formula (\ref{EF}) for $\alpha$ equals $\mrm{Ind}_t(d_{\mathcal{F}})(\alpha)$. In this section, we will interpret the left hand term of (\ref{EF}) in terms of the dynamical system $(S(A_0), \mathcal{F}, \phi^t)$ by proving the existence of a natural bijection between the set of primitive compact orbits of $\phi^t$ and the set of closed points of the abelian variety $A_0$. Furthermore, according to Deninger's analogy \cite{deninger3}, if a primitive compact orbit $\gamma$ corresponds to the closed point $x$, then the length of $\gamma$ should be $\log N(x)$ where $N(x)=q^{\deg(x)}$ denotes the cardinality of the residue field of $x$. To prove these facts, we need a preliminary lemma.\\

For any compact orbit $\gamma$ of $\phi^t$ on $S(A_0)$, denote by $l(\gamma)$ the length of $\gamma$. Note that if $[z, \hat{v}, x] \in S(A_0)$ is such that $\phi^t[z, \hat{v}, x]=[z, \hat{v}, x]$ then necessarily $t=\nu \log q$ for some integer $\nu \geq 0$. This shows that the length of any compact orbit is of the form $\nu \log q$ for some integer $\nu \geq 0$. 

\begin{lem} \label{technique-pour-noeuds1} There is a natural bijection between the set of primitive compact orbits of $\phi^t$ of length $\nu \log q$ and the set of $\xi$-orbits on $\C^g/\Gamma$ of order $\nu$.
\end{lem}

\begin{proof} Let us denote by $\overline{M}$ be the quotient space $\C^g \times_V V_\xi \Gamma$. Then $q^{\Z}$ acts on $\overline{M}$ by $q^\nu[z, \hat{v}]=[\xi^{-\nu}_{\R}z, \xi^{-\nu} \hat{v}]$. Identify $\overline{M}$ with its image $\overline{M} \times \{1\}$ in $S(A_0)$. Then the map $\gamma \longmapsto \gamma \cap \overline{M}$ is a bijection between the set of primitive compact orbits of $\phi^t$ on $S(A_0)$ of length $\nu \log q$ and the finite orbits of $q^{-1}$ on $\overline{M}$ of order $\nu$. According to Lem. \ref{technique}, there is a natural bijection $\C^g \times_\Gamma T_\xi \Gamma \simeq \C^g \times_V V_\xi \Gamma$ which is equivariant with respect to the diagonal action of $\xi$ on both sides. Furthermore, under the projection $\C^g \times_\Gamma T_\xi \Gamma \longrightarrow \C^g/\Gamma$, the orbits of order $\nu$ of the diagonal action of $\xi$ are mapped bijectively onto the $\xi$-orbits of order $\nu$ on $\C^g/\Gamma$. The inverse map sends the orbit of $z+\Gamma$ to the orbit of $[z, \hat{\gamma}]$ where $\hat{\gamma}=(1-\xi^\nu)^{-1} \gamma$ if $\gamma=\xi^{\nu}z-z \in \Gamma$. Note that this makes sense because $1-\xi^{\nu}$ is invertible on $T_\xi \Gamma$: indeed if we denote by $V$ the Verschiebung endomorphism of $\Gamma$ we have $(1-\xi^\nu)^{-1}=V^{\nu}(1-q^{\nu})^{-1}$.
\end{proof}

The proof of the next result is a direct generalisation of the one of \cite{deninger} Prop. 3.3. 

\begin{pro} \label{noeuds}  There is a natural bijection between the set of closed points of $A_0$ and the set of primitive compact orbits of $\phi^t$ on $S(A_0)$ such that if $x$ corresponds to $\gamma$, then 
$$
l(\gamma)=\log N(x),
$$
where $N(x)=q^{\deg(x)}$ denotes the cardinality of the residue field of $x$.
\end{pro}

\begin{proof} The closed points $x$ of $A_0$ such that $N(x)=n$ are in bijection with the $\phi_0$-orbits on $A_0(\overline{\F}_q)$ of order $n$. Hence, according to Lem. \ref{technique-pour-noeuds1}, it remains to construct a natural bijection between the $\xi$-orbits of order $n$ on $\C^g/\Gamma$ and the $\phi_0$-orbits of the same order on $A_0(\overline{\F}_q)$. Let $\overline{K}$ be the algebraic closure of $K$ in $\C$ and let $\overline{W}$ be the integral closure of $W$ in $\overline{K}$. Fix a maximal ideal $\overline{\mathfrak{m}}$ over $\mathfrak{m}$. Then $\overline{W}/\overline{\mathfrak{m}}=\overline{\F}_q$. By the valuative criterion of properness, the natural map $\mathcal{A}(\overline{W}) \longrightarrow A(\overline{K})$ is an isomorphism. As the torsion points of $A(\C)$ are algebraic over $K$, the inclusion $A(\overline{K}) \longrightarrow A(\C)$ induces an isomorphism on the torsion subgroups. As a consequence, we obtain a reduction map
$$
\mrm{red}: A(\C)_{tors} \simeq \mathcal{A}(\overline{W})_{tors} \longrightarrow \mathcal{A}(\overline{W}/\overline{\mathfrak{m}})_{tors}=A_0(\overline{\F}_q).
$$
This map is equivariant with respect to the actions of $\xi$ on the left and $\phi_0$ on the right. The group
$$
(A(\C)_{tors})^{\xi^n=1}=A(\C)^{\xi^n=1}=(\C^g/\Gamma)^{\xi^n=1}=(\xi^n-1)^{-1} \Gamma/\Gamma=\Gamma/(\xi^n-1)\Gamma
$$
is finite and has order $\det(\xi^n-1)$. The group $A_0(\overline{\F}_q)^{\phi_0^n=1}=A_0(\F_{q^n})$ has the same order according to the first statement of \cite{mumford} IV \S 21 Thm. 4. Moreover, for any integer $N$ prime to $p$, the restriction
$$
\mrm{red}: A(\C)_N \longrightarrow A_0(\overline{\F}_q)_N
$$
is an isomorphism (see \cite{deligne} (3.1)). As $\mathcal{A}/W$ is the Serre-Tate canonical lift of $A_0/k$, we have
$$
\mathcal{A}(\overline{W})_{p^n}=\mathcal{G}^0_{p^n} \times (\Z/p^n \Z)^g
$$
for some group $\mathcal{G}^0_{p^n}$ and where the natural projection $\mathcal{A}(\overline{W})_{p^n} \longrightarrow (\Z/p^n \Z)^g$ is identified to the reduction
$$
\mrm{red}: \mathcal{A}(\overline{W})_{p^n} \longrightarrow \mathcal{A}(\overline{W}/\overline{\mathfrak{m}})_{p^n}=A_0(\overline{\F}_q)_{p^n} \simeq (\Z/p^n \Z)^g.
$$
As a consequence, via the identification $\mathcal{A}(\overline{W})_N=A(\overline{K})_N$ we obtain
$$
A(\overline{K})_N=\mathcal{G}^0_{p^{v_p(N)}} \times A_0(\overline{\F}_q)_N
$$
for any integer $N \geq 0$. Taking $N=\det(\xi^n-1)=|A_0(\F_{q^n})|$ and passing to $\phi^n-1$ fixed modules we obtain a natural surjection $$A(\C)^{\xi^n=1}=A(\overline{K})^{\phi^n=1} \longrightarrow A_0(\F_{q^n})$$ which is a bijection because the left and right hand groups have the same order.
\end{proof}

Combining (\ref{EF}), Prop. \ref{eigenvalues} and Prop. \ref{noeuds} we obtain the following index theoretic way to write the explicit formula for an ordinary abelian variety of dimension $g$. 

\begin{cor} \label{main} The following equality holds in the space of distributions $\mathcal{D}'(\R)$:
$$
\mrm{Ind}_t(d_{\mathcal{F}})=\sum_{\gamma} l(\gamma) \sum_{k \geq 1} \delta_{k l(\gamma)}+\sum_{\gamma} l(\gamma) \sum_{k \leq -1} e^{gkl(\gamma)}\delta_{kl(\gamma)}.
$$
\end{cor}

\begin{rems} There is a perfect analogy between the formula above and the transversal index formula \cite{deninger} Thm. 2.2 except for the factor $e^{gkl(\gamma)}$ in front of Dirac distribution $\delta_{kl(\gamma)}$ for $k \leq -1$. For an explanation of this dissymetry, see Rem (2) p. 213 in \cite{leichtnam1} and \cite{deninger1} p. 18. 
\end{rems}

\end{document}